\documentclass[preprint,12pt]{elsarticle}

\usepackage[utf8]{inputenc}
\usepackage[english]{babel}
\usepackage{amsthm}
\usepackage{amsmath}
\usepackage{amsfonts}
\usepackage{appendix}
\usepackage[normalem]{ulem}
\usepackage{mathtools}
\usepackage{stmaryrd}
\usepackage{mathabx}
\usepackage{bigints}
\usepackage{caption}

\DeclareCaptionLabelSeparator{none}{ }
\usepackage[left=2.5cm, right=2.5cm, top=2.5cm, bottom=2.5cm]{geometry}

\usepackage[inline, shortlabels]{enumitem}

 \usepackage[hidelinks, bookmarks, bookmarksnumbered, pdfstartview={XYZ null null 1.00}]{hyperref}

\newtheorem{theorem}{Theorem}[section]
\newtheorem{lemma}[theorem]{Lemma}
\newtheorem{proposition}[theorem]{Proposition}

\theoremstyle{definition}
\newtheorem{definition}[theorem]{Definition}
\newtheorem{remark}[theorem]{Remark}

\DeclareMathOperator{\rank}{rank}

\DeclareMathOperator{\supp}{supp}

\DeclarePairedDelimiter{\abs}{\lvert}{\rvert}
\DeclarePairedDelimiter{\scalprod}{\langle}{\rangle}

\newcommand{\vertiii}[1]{{\left\vert\kern-0.25ex\left\vert\kern-0.25ex\left\vert #1 
        \right\vert\kern-0.25ex\right\vert\kern-0.25ex\right\vert}}

\newcommand{\vertii}[1]{{\left\vert\kern-0.25ex\left\vert #1 
        \right\vert\kern-0.25ex\right\vert}}

\newcommand{\suchthat}{\ifnum\currentgrouptype=16 \mathrel{}\middle|\mathrel{}\else\mid\fi}

\theoremstyle{definition}

\begin{document}

\begin{frontmatter}



\title{Left-coprimeness condition for the reachability in finite time of pseudo-rational systems of order zero with an application to difference delay systems\tnoteref{tf} } 

%

\author{S\'ebastien Fueyo\corref{ca}} 

\cortext[ca]{corresponding author}


\affiliation{organization={Univ. Grenoble Alpes, Inria, CNRS, Grenoble INP, GIPSA-Lab},
            city={Grenoble},
            postcode={F-38000}, 
            country={France}}

\ead{sebastien.fueyo@gipsa-lab.fr}
            
\begin{abstract}
This paper investigates the finite-time reachability of pseudo-rational systems of order zero. A bound on the minimal time for reachability is derived, and the reachability property for integrable functions is characterized using a left-coprimeness condition. The results are further applied to difference delay systems with distributed delays.
\end{abstract}




\end{frontmatter}

\section{Introduction}

\textbf{Pseudo-rational systems and controllability of delay equations.} The study of so-called pseudo-rational input-output systems, a concept first introduced by Y. Yamamoto in \cite{YamamotoRealization}, is important because it represents a large class of systems which contains in particular delay systems; the latter frequently appearing in a wide range of applications such as electronic engineering \cite{brayton1968small} and biology \cite{belair1987model}. In particular, the control and stabilization properties of delay systems have been extensively studied in the literature \cite{Manitius,connor,Jacobs,hale2002stabilization}, with a primary focus on neutral differential delay equations:
 \begin{equation}
    \label{eq:2}
    \textstyle\frac{\mathrm{d}}{\mathrm{d}t} \left(x(t)- \sum\limits_{j=1}^N A_j \,x(t-\Lambda_j)\right)
    =\sum\limits_{j=1}^NC_j\,x(t-\Lambda_j)+Bu(t),\qquad t\geq 0.
  \end{equation}
The state \( x(t) \) and the control \( u(t) \) belong to \( \mathbb{R}^d \) and \( \mathbb{R}^m \), respectively, where \( d \) and \( m \) are positive integers, and the system's delays are distinct and ordered as \( 0 < \Lambda_1 < \dotsc < \Lambda_N \). Equation~\eqref{eq:2} is a generalization of the retarded differential delay systems ($A_j \equiv 0$, $\forall j \in \{1,\cdots,N\}$), see the book \cite{Hale} for the classification of delay systems.

Neutral functional differential equations, being inherently infinite-dimensional, give rise to various concepts of controllability, including approximate, exact, and relative controllability. These concepts can be further refined depending on whether controllability is expected to occur within a uniform time or not. Manitius \cite{Manitius} established approximate controllability criteria for retarded differential delay systems in the frequency domain. These criteria are of the Hautus type, \textit{i.e.}, they follow the Hautus criterion for finite-dimensional systems \cite{hautus1969controllability}. O'Connor and Tarn \cite{connor} extended this result to neutral differential equations with a single delay, but their approach was limited to that specific case.
Yamamoto \cite{yamamoto1989reachability} later proposed a framework for neutral control systems with multiple delays by interpreting them as pseudo-rational systems. This approach enabled him to generalize the frequency domain criteria for approximate controllability obtained by \cite{connor}.  However, the exact controllability of systems like \eqref{eq:2} with multiple delays remains largely unexplored in the literature.

Pseudo-rational systems were originally introduced in \cite{YamamotoRealization}. Throughout this paper, the term "reachability" will often replace "controllability" to align with realization theory and input-output systems terminology, although the two terms are equivalent \cite{yamamoto1989reachability}. A pseudo-rational system is said to be \textit{reachable or quasi-reachable} if, for any target belonging in a certain space, there is an input that allows it to be reached or approached. Furthermore, a pseudo-rational system is \textit{reachable or quasi-reachable in finite time $T>0$} if we can reach all the targets with inputs applied during a time $T$. Definitions of quasi-reachability and reachability for pseudo-rational systems discussed in this article are formalized in Definition~\ref{def_reachability}. The infimum of the positive real $T$ satisfying the property above is called \textit{the minimal time of reachability or quasi-reachability}.

The article \cite{YamamotoRealization} characterized the reachability of pseudo-rational systems using a left-coprimeness condition, which was shown to be both necessary and sufficient when the targets belong to the function space of square-integrable functions $L^2$ with respect to the Lebesgue measure. A (approximate) left-coprimeness condition is also known as a (approximate) Bézout identity in control theory. Remarkably, (approximate) controllability can be characterized through the solvability of a (approximate) Bézout identity using input-output systems.
In addition, algebraic and frequency domain criteria for quasi-reachability are stated in the paper \cite{yamamoto1989reachability} by exploiting the left-coprimeness. The results are applied to neutral differential delay systems \eqref{eq:2} to obtain the approximate controllability frequency domain criteria mentioned above. On the other hand, reachability has been less studied, and the most complete exposition in the literature is given in \cite{YAMAMOTO201620} for scalar pseudo-rational systems. A necessary and sufficient reachability criterion for targets belonging to a distribution space, based on a left-coprimeness condition, is given, as well as a frequency domain criterion. However, the link between the reachability notion and the targets belonging to function spaces is not investigated. Unfortunately, it is the reachability in function spaces which is of interest for practical applications. Moreover, quasi-reachability or reachability in finite time has not been properly studied.

\medskip

\textbf{Reachability: pseudo-rational systems of order zero, difference delay systems and 1-D hyperbolic PDEs.} Recently, \cite{Chitour2020Approximate} reinvested the study of the so called difference delay systems of the form

\begin{equation}
\label{eq:diff_delay_without}
   x(t)= \sum_{j=1}^N A_j \,x(t-\Lambda_j)
     +Bu(t),\qquad t\geq 0.
  \end{equation}
These systems are of particular interest due to their connections with one-dimensional (1-D) hyperbolic partial differential equations (PDEs), such as linear conservation laws. Through the method of characteristics \cite{baratchart,bastin2016stability,Chitour2016Stability,CoNg}, these PDEs can be transformed into the form~\eqref{eq:diff_delay_without}, resulting in a system where the matrices $A_i$ exhibit a specific structure, notably all being rank-one matrices. Studying difference delay systems has been crucial for deriving properties, such as the stability and the controllability, for 1-D hyperbolic PDEs.
  
  The main results of \cite{Chitour2020Approximate}  regarding the exact controllability of \eqref{eq:diff_delay_without} include an algebraic criterion for exact controllability in $L^2$, but only for a low-dimensional state space and a single input. Additionally, they fully characterized the minimal controllability time. However, the methods used, based on observability inequalities, were not suitable for higher dimensions ($d>2$) with multiple inputs ($m>1$). This limitation led to renewed interest in algebraic methods from realization theory and input-output systems, see \cite{chitour:hal-03827918}. Specifically, System~\eqref{eq:diff_delay_without} can be interpreted as a pseudo-rational system of order zero, defined by the absence of derivatives in the equation (unlike a neutral system), see Definition~\ref{def_pseudo_rational}. In \cite{chitour:hal-03827918}, the detailed analysis of pseudo-rational systems of order zero in the context of controllability for difference delay systems \eqref{eq:diff_delay_without} led to the following three significant results:
  \begin{enumerate}
      \item Frequency domain criteria for approximate controllability derived from \cite{yamamoto1989reachability}.
      \item An upper bound on the minimal time of controllability, given by $d \Lambda_N$, where $\Lambda_N$ is the largest delay and $d$ is the dimension of the vector $x$.
      \item Characterization of $L^1$ exact controllability through a left-coprimeness condition, where $L^1$ denotes a space of integrable functions.
  \end{enumerate}
In \cite{fueyo-chitour}, the left-coprimeness condition was solved to establish $L^1$ frequency domain criteria for exact controllability, while \cite{chitour:hal-04228797} addressed controllability issues for certain 1-D hyperbolic PDEs. Notably, the Hautus criteria derived differ from those for abstract differential equations \cite{Tucsnak_Weiss}, as they are not based on abstract differential representations of System~\eqref{eq:diff_delay_without}.

\medskip

\textbf{Results of the paper and comments.} This paper investigates the minimal time of reachability for pseudo-rational systems of order zero, with a focus on targets that belong to function spaces. A key aspect of the paper is the application of these systems to difference delay systems with distributed delays, which are defined by the following equation:
\begin{equation}
\label{system_lin_formel2bis}
 x(t)=\sum_{j=1}^NA_jx(t-\Lambda_j)+\int_0^{\Lambda_N} g(s) x(t-s)ds+Bu(t) , \qquad t \ge 0.
\end{equation}
When $g \equiv 0$, System~\eqref{system_lin_formel2bis} reduces to a difference delay system \eqref{eq:diff_delay_without}. The introduction of distributed delays renders the computational methodology used in \cite{chitour:hal-03827918} ineffective due to the presence of the integral. It is also worth noting that System~\eqref{system_lin_formel2bis} is of particular interest for hyperbolic systems \cite{auriol2019explicit}. 

The major results of the paper are threefold. First, we prove that quasi-reachability and reachability are equivalent to their quasi-reachability and reachability in finite time, and we are able to provide an upper bound on the minimal time of reachability, see Theorem~\ref{bound_minimal_time_controllability} in Section~\ref{sec_min_time}. A similar result has been obtained in the scalar case ($m=d=1$), see \cite{YAMAMOTO2021353}. Theorem~\ref{bound_minimal_time_controllability}  differs in that it considers the matrix case ($m,d>1$).

The second result, presented in Theorem~\ref{th_coprm_L1} of Section~\ref{sec4}, gives a necessary and sufficient condition for reachability, using a left-coprimeness property, when the targets belong to \(L^1\). This result, based on the minimal time of reachability, is new from an algebraic perspective, as reachability has previously been studied when the input and output belong to distribution spaces \cite{kamen1976,YAMAMOTO201620}. Theorem~\ref{th_coprm_L1} is of importance because it provides a finer characterization of reachability in function spaces, as compared to distribution spaces.
In addition, the left-coprimeness condition provides an effective procedure to construct inputs that reach any target in the function space \( L^1 \); see Remark~\ref{remark_build2}.

The approach of pseudo-rational systems of order zero provides an upper bound on the minimal time for the quasi-reachability and reachability of difference delay systems with distributed delays, as described in equation \eqref{system_lin_formel2bis}, with the bound being \( d \Lambda_N \) (see Theorem~\ref{theorem_reachability_dis_delay}). In addition, it offers a necessary criterion for reachability in the frequency domain, of the Hautus type, for targets in \( L^1 \), as shown in Theorem~\ref{th:main_result_dist}. Specifically, Theorem~\ref{theorem_reachability_dis_delay} and Theorem~\ref{th:main_result_dist} generalize the results of \cite{chitour:hal-03827918} to the more complex case of distributed delays in difference delay systems. Furthermore, the upper bound \( d\Lambda_N \) is directly derived from Theorem~\ref{bound_minimal_time_controllability}, whose proof is notably simpler than the one in \cite{chitour:hal-03827918}. The frequency domain criterion for reachability presented in Theorem~\ref{th:main_result_dist} does not impose any rational dependence assumption on the delay, which is a common assumption in time-delay system theory. To conclude, whether the necessary frequency domain criterion for \( L^1 \) reachability is also sufficient remains an open question, as resolving this would require addressing a corona problem that has yet to be solved in the literature.

\medskip

\textbf{Organization of the paper.} The remaining structure of the paper is as follows. Section~\ref{sec:not2} introduces the notation and fundamental concepts, while Section~\ref{sec:pseudo-rational3} defines pseudo-rational systems of order zero and explores their key properties. In Section~\ref{sec_min_time}, we investigate the concept of minimal time of reachability and its implications. Section~\ref{sec4} focuses on the left-coprimeness conditions for reachability, establishing their connection to function spaces. Finally, in Section~\ref{sec:appli_delay7}, we apply the theoretical results to difference delay systems with distributed delays, demonstrating the broader applicability of our findings.

\section{Notation}
\label{sec:not2}
The sets of natural numbers, real numbers and
complex numbers are denoted by $\mathbb{N},\mathbb{R}$ and $\mathbb{C}$. For $p \in \mathbb{C}$, $\Re(p)$ denotes the real part of $p$. The subspaces $\mathbb{R}_-$ and $\mathbb{R}_+$ of $\mathbb{R}$ stand for the intervals $(-\infty,0]$ and $[0,+\infty)$. For an integer $n$, $g \in \mathbb{C}^n$ defines a column vector, and we use $g^T$ to denote its transpose, which represents a row vector. The transpose here is simply a practical tool to express row vectors instead of column vectors and should not be confused with the conjugate transpose.  We endow $\mathbb{C}^n$ with any norm $\|\cdot\|$. Since all norms are equivalent in finite dimensions, the choice of a specific norm is irrelevant for the convergence in the function spaces considered below. For all square matrix $F$ with values in a ring, we use $\det(F)$ for the determinant of $F$ and $I_d$ denotes the identity matrix. For a function $f$ defined in $\mathbb{R}$, we denote by $\pi(f):=f_{|\mathbb{R}_+}$ the truncation of the function $f$ to non-negative times. 

For $q \in [1,+\infty)$ and $K$ a compact of $\mathbb{R}$,
$L^q(K,\mathbb{R})$ represents the set of Lebesgue measurable functions for which the $q$-th power of the absolute value is Lebesgue integrable, i.e.
\[
\|\phi \|_{K,q}:=\left(\int_K \abs{\phi(t)}^q dt  \right)^{1/q}<+\infty,\qquad \phi \in L^q(K,\mathbb{R}).
\]
In the remaining of the paper, we interpret the space $L^q(K,\mathbb{R})$ as a subspace of $L^q$-functions on $\mathbb{R}$ that vanish outside the set $K$. We write $L^q_{\rm loc}(\mathbb{K},\mathbb{R})$ the spaces of $q$-th integrable functions on each compact of $\mathbb{K}$, where $\mathbb{K}=\mathbb{R}_-$ or $\mathbb{R}_+$. The family of the semi-norms $\|\cdot\|_{K,q}$, for $K$ compact interval of $\mathbb{R}$, induces a topology on $L^q_{\rm loc}(\mathbb{K},\mathbb{R})$ which is then a Fréchet space. For some references on Fréchet space, see \cite[Definition 2.4, p.106]{conway2019course}. Let $\Omega_q:= \bigcup_{i=1}^{\infty} L^q([-i,0],\mathbb{R})$ endowed with the topology of the strict inductive limit of the sequence $\{L^q([-i,0],\mathbb{R}) \}_{i=1}^{\infty}$. The space $\Omega_q$ can be interpreted as the functions belonging to $L^q_{\rm loc}(\mathbb{R}_-,\mathbb{R})$ and having compact supports contained in $\mathbb{R}_-$, see \cite{YamamotoRealization} and \cite[appendix]{Yutaka_Yamamoto}.

For the fundamental concepts of the  theory of distributions, the reader may refer to Schwartz's book \cite{schwartz1966theorie}. Let $\mathcal{D}'(\mathbb{R})$ the space of distributions on $\mathbb{R}$ endowed with its usual topology. The support of a distribution always exists \cite[p.28]{schwartz1966theorie}, and for \( \alpha \in \mathcal{D}'(\mathbb{R}) \), \( \mathrm{supp}(\alpha) \) denotes the support of \( \alpha \). By the completeness of the real numbers, since the support of a distribution is a subset of \( \mathbb{R} \), its infimum necessarily exists (though it can be equal to \( -\infty \)). Thus, for \( \alpha \in \mathcal{D}'(\mathbb{R}) \), we denote the infimum of the support of \( \alpha \) by
\begin{equation}
\label{def_alpha}
l(\alpha) := \inf \, \mathrm{supp} (\alpha).
\end{equation}
In particular, the support of a distribution coincides with the support of a function, so that for a function \( f : \mathbb{R} \to \mathbb{R} \), we have \( l(f) \), which is the infimum of the support of \( f \).

The canonical duality pairing between a distribution in $\mathcal{D}'(\mathbb{R})$ and a function in $\mathcal{D}(\mathbb{R})$ is denoted $\scalprod{.,.}$, where $\mathcal{D}(\mathbb{R})$ is the space of $C^{\infty}$-functions defined on $\mathbb{R}$ with the usual topology. The convolution between two distributions, when it exists, is denoted by $*$ and the Laplace transform of a distribution $\alpha \in \mathcal{D}'(\mathbb{R})$ is written $\hat{\alpha}$.
The subspaces of $\mathcal{D}'(\mathbb{R})$, denoted $\mathcal{E}'(\mathbb{R}_-)$ and $\mathcal{E}'(\mathbb{R}_+)$, stand for the distributions with compact support included in $\mathbb{R}_-$ and  $\mathbb{R}_+$. The distributions of order zero with compact support in $\mathbb{R}_-$ and support bounded on the left are denoted by $M(\mathbb{R}_-)$ and $M_+(\mathbb{R})$ respectively. The distributions of order zero are easily identified with Radon measures by the Riesz representation theorem, see \cite[Theorem 5.7, p.75]{conway2019course}. The Dirac distribution (of order zero) at $t \in \mathbb{R}$ is denoted $\delta_t$.

Following \cite{YamamotoRealization,yamamoto1989reachability}, we extend the truncation operator $\pi$ for distributions. Let $\mathcal{D}(\mathbb{R}_+)$ be the space of infinitely differentiable functions on $\mathbb{R}$ with compact support contained in $\mathbb{R}_+$. For $\alpha \in \mathcal{D}'(\mathbb{R})$, we define the operator $\pi$ as:
\begin{equation}
\label{eq:pi}
\scalprod{\pi \alpha, \psi} = \scalprod{\alpha, \psi},\qquad \psi \in \mathcal{D}(\mathbb{R}_+).
\end{equation}
We have that $\pi \alpha$ is a well-defined element of $\mathcal{D}'(\mathbb{R}_+)$ the dual space of $\mathcal{D}(\mathbb{R}_+)$. In particular if $\pi(\alpha)=0$ then the support of $\alpha$ is included in $\mathbb{R}_-$.

 With a slight abuse of language, we keep the former notations introduced when dealing with vectors or matrices whose entries belong to the spaces formerly introduced.

\section{Pseudo-rational systems of order zero}
\label{sec:pseudo-rational3}

In this section, we recall basic facts about pseudo-rational systems of order zero, a concept first introduced by Y. Yamamoto \cite{YamamotoRealization}. The only minor difference in our exposition is the consideration of $L^q$ spaces, for all $q\in [1,+\infty)$ instead of $q=2$, but the theory developed in \cite{YamamotoRealization} remains valid in this framework. In the sequel of this paper, unless otherwise stated, $q$ always denotes a real number belonging to $[1,+\infty)$.

Let $A$ be a Radon measure with compact support in $\mathbb{R}_-$. Let us consider a linear (zero-initial state) input-output system, denoted $\Sigma$, defined by 
the impulse response matrix\footnote{In linear systems theory, as discussed in \cite{kailath1980linear}, the impulse response matrix is defined as the system's output when the input is the Dirac delta distribution at time zero. In other words, if $u=\delta_0$, the system's output is the matrix $A$.} $A$ such that the input and the output of $\Sigma$ are linked as follows:
\begin{equation}
\label{eq:premiere}
    y=S_q(u), \quad S_q(u)=\pi(A*u),
\end{equation}
where $u$ and $y$ belong to the input state space $(\Omega_q)^m$ and the output state space $(\Gamma_q)^d$, with $\Gamma_q:=L^q_{\rm loc} (\mathbb{R}_+,\mathbb{R})$. By \cite[Section 6.1]{schwartz1966theorie}, it is straightforward to see that $S_q:=\pi \left(A*u\right)$ is a well-defined continuous linear operator from $(\Omega_q)^m$ to $(\Gamma_q)^d$, where the topologies on $(\Omega_q)^m$ and $(\Gamma_q)^d$ are the product topologies,  ensuring that $\Sigma$ is well-defined. In this paper, we consider $\Sigma$ systems with particular impulse response matrices.

\begin{definition}
\label{def_pseudo_rational}
   A $\Sigma^{Q}$ system is said to be \textit{pseudo-rational of order zero} if
    there are two  $d\times d$ and $d \times m$ matrices with entries in $M(\mathbb{R}_-)$, denoted $Q$ and $P$ respectively, such that
     \begin{enumerate}
        
\item \label{Item_def_pseudo_rational1}$Q$ has an inverse over $M_+(\mathbb{R})$ in a convolution sense, \textit{e.g.}, there exists a $d \times d$ matrix with entries in $M_+(\mathbb{R})$ such that $Q^{-1}*Q=Q*Q^{-1}=\delta_0 I_d$;
\item \label{Item_def_pseudo_rational2} $A:=Q^{-1}*P$.
        \end{enumerate}
\end{definition}

\begin{remark}
The term \textit{pseudo-rational system of order zero} stems from the fact that \( Q \) and \( P \) are Radon measures, which are distributions of order zero. Typical examples of pseudo-rational systems of order zero include difference delay systems, as seen in Equation~\eqref{eq:diff_delay_without}.
\end{remark}

We introduce the state space of pseudo-rational systems of order zero as $X^{Q,q}$, the subspace of $(\Gamma_q)^d$, defined as follows:
\begin{equation}
    X^{Q,q}:= \left\{y \in (\Gamma_q)^d \suchthat \pi(Q*y)=0  \right\}.
\end{equation}
 For $T > -l(\det(Q))$, we endow $ \Gamma_{q}$ with the norm $\|\cdot\|_{L^q([0,T],\mathbb{R})}$. Denoting by $\|\cdot\|_{X^{Q,q}_T}$ the product norm on $(\Gamma_q)^m$, we have that $X^{Q,q}$ is a Banach space when endowed with this norm, see \cite[Proposition 3.5]{YamamotoRealization}. From \cite[Theorem 2.9]{YamamotoRealization}, we have that $\overline{\mathrm{Im} S_q} \subset X^{Q,q}$, where $\overline{\mathrm{Im} S_q} $ is the closure of the image of the operator $S_q$. 
 
 We define the quasi-reachability and reachability\footnote{The quasi-reachability and the reachability are referred to as approximate controllability and controllability in a control theory formulation, see \cite{yamamoto1989reachability}.} properties within this formalism for a $\Sigma^Q$ system. Since the quasi-reachability depends on the convergence of sequences, we endow the space $(L^q_{\rm loc}(\mathbb{R}_+,\mathbb{R}))^m$ with its product topology, where the topology of $L^q_{\rm loc}(\mathbb{R}_+,\mathbb{R})$ is introduced in Section~\ref{sec:not2}.

\begin{definition}
\label{def_reachability}
    A $\Sigma^Q$ pseudo-rational system of order zero is:
    \begin{enumerate}[i)]
        \item 
    \textit{$X^{Q,q}$ quasi-reachable} if for all $y \in  X^{Q,q}$ there exist a sequence $(u_n)_{n \in \mathbb{N}} \in (\Omega_q)^m$ such that 
    \begin{equation}
        \pi(Q^{-1}*P*u_n) \longrightarrow y,\quad\mbox{in $(L^q_{\rm loc}(\mathbb{R}_+,\mathbb{R}))^m$ as $n \rightarrow +\infty.$}
    \end{equation}
   \item \textit{$X^{Q,q}$ quasi-reachable in finite time $T>0$} if for all $y \in  X^{Q,q}$ there exist a sequence $(u_n)_{n \in \mathbb{N}} \in (\Omega_q)^m$ with support in $[-T,0]$ such that 
    \begin{equation}
        \pi(Q^{-1}*P*u_n) \longrightarrow y,\quad\mbox{in $(L^q_{\rm loc}(\mathbb{R}_+,\mathbb{R}))^m$ as $n \rightarrow +\infty.$}
   \end{equation}
\item 
    \textit{$X^{Q,q}$ reachable} if for all $y \in  X^{Q,q}$ there exist a $u\in (\Omega_q)^m$ such that 
    \begin{equation}
        \pi(Q^{-1}*P*u) =y.
    \end{equation}
\item 
    \textit{$X^{Q,q}$ reachable in finite time $T>0$} if for all $y \in  X^{Q,q}$ there exist a $u\in \Omega_q^m$ with support in $[-T,0]$ such that 
    \begin{equation}
        \pi(Q^{-1}*P*u) =y.
    \end{equation}
        \end{enumerate}
\end{definition}

One fundamental challenge in establishing finite-time reachability is to provide an upper bound on the time required for it to occur.

\begin{definition} We define $T_{ \rm min,q}^{\rm qr}$ and $T_{ \rm min,q}^{\rm r}$ as the \textit{minimal time} of the $X^{Q,q}$ quasi-reachability and reachability for pseudo-rational system of order zero $\Sigma^Q$ as follows:
\begin{align*}
T_{ \rm min,q}^{\rm qr}&:=\inf_{T >0} \left\{ \mbox{$\Sigma^Q$ is $X^{Q,q}$ quasi-reachable in time $T$}\right\}.\\
T_{ \rm min,q}^{\rm r}&:=\inf_{T>0} \left\{ \mbox{$\Sigma^Q$ is $X^{Q,q}$ reachable in time $T$}\right\}.
\end{align*}
\end{definition}

The remainder of this paper explores the conditions for quasi-reachability and reachability, as well as the minimal time required for reachability in $\Sigma^Q$ systems.

\section{Minimal time of reachability}
\label{sec_min_time}

This section is dedicated to the equivalence between the reachability notions introduced in Definition~\ref{def_reachability} and to providing an upper bound on the minimal time of reachability. 
Since we will extensively use $l$  in the sequel of this section, we recall that $l$  defines the infimum of the support of a distribution and of a function, as stated in Equation~\eqref{def_alpha}. We present the following fundamental result.

\begin{theorem} 
\label{bound_minimal_time_controllability}
A $\Sigma^Q$ pseudo-rational system of order zero is $X^{Q,q}$ quasi-reachable (resp. reachable) if and only if it is  $X^{Q,q}$ quasi-reachable (resp. reachable) in time $\widetilde{T}$ for all $\tilde{T}>-l(\det(Q))$. In particular, an upper bound on the minimal time of quasi-reachability (resp. reachability) is given by $-l(\det(Q))$.
\end{theorem}

\begin{remark}
\label{remark_further_details}
     The result of Theorem~\ref{bound_minimal_time_controllability} was initially established for the scalar case ($m = d = 1$) and later extended to the matrix case for the algebra of distributions with compact support in $\mathbb{R}_-$; see \cite{Yamamoto_coprimness_measure} and \cite{YAMAMOTO_Multi_Ring_2011}. The proof involves analyzing the dual of a pseudo-rational system and demonstrating its observability within a finite time, a result that was first obtained in \cite{yamamoto1984note} with an explicit bound. A proof for the scalar case with Radon measures was subsequently provided in \cite{YAMAMOTO2021353}.
\end{remark}

\begin{remark}
    Kamen studied the minimal time of reachability for general input-output systems \cite{kamen1976}, where \( A \) given in \eqref{eq:premiere} is a distribution (with support bounded on the left), and the input and output spaces are also distributions. He provided an upper bound on the minimal time of reachability for these systems, but did not address the issue of pseudo-rational systems when the input and output spaces are functional.
\end{remark}

 Before the proof of Theorem~\ref{bound_minimal_time_controllability}, we begin with an adaptation of the Kamen lemma \cite[Lemma 6.1]{kamen1976} to align with the framework of pseudo-rational systems of order zero.

\begin{lemma}
\label{lemma_Kamen}
Let $\beta \in M(\mathbb{R}_-)$ having an inverse denoted $\beta^{-1}$ in $M_+(\mathbb{R})$ and $ \tau < l(\beta)$. For all $\omega \in \Omega_q$, there exist $\alpha,r \in \Omega_q$ such that $\omega=\alpha+\beta*r$ and  $l(\alpha)>\tau$.
\end{lemma}

\begin{proof}
 Pick $\omega \in \Omega_q$. If $l(\omega)>\tau$, the result follows immediately. Thus, we assume that $l(\omega) \le \tau$. Using the equality $\beta *\left( \beta^{-1}*\omega \right)= \omega$, we deduce that $l(\omega) \ge l(\beta)+l(\beta^{-1}*\omega)$, which implies
\begin{equation}
    l(\beta^{-1}*\omega) \le \tau -l(\beta)<0,
\end{equation}
thanks to the inequalities $l(\omega) \le \tau$ and $\tau <l(\beta)$, as well as the continuity of the infimum.  Let $a_1,a_2,b_1,b_2 \in \mathbb{R}$ such that $-\infty<a_2<a_1<l(\beta^{-1}*\omega)$ and $\tau-l(\beta)<b_1<b_2<0$. By the Tietze extension theorem \cite[Theorem 20.4]{Rudin1987Real}, there exists a continuous function on $\mathbb{R}$ with compact support, denoted $\phi$, such that $\phi$ is equal to $1$ on $[a_1,b_1]$, $0$ on the complementary of $[a_2,b_2]$ in $\mathbb{R}$ and $0 \le \phi(t) \le 1$ for all $t\in \mathbb{R}$. Define $r:=(\beta^{-1}*\omega)\phi$ the pointwise product of $\beta^{-1}*\omega$ and $\phi$. Since the support of \( \phi \) is contained in \( [a_2, b_2] \) by definition, the support of \( r \) is also contained in \( [a_2, b_2] \), as \( \phi \) vanishes outside this interval. Furthermore, by construction, \( a_2 < l(\beta^{-1} * \omega) \), so \( r \) is zero on \( [a_2, l(\beta^{-1} * \omega)) \) because \( \beta^{-1} * \omega \) vanishes on the same interval. As \( l(\beta^{-1} * \omega) < b_2 \), it follows that the support of $r$ is contained in $[l(\beta^{-1}*\omega),b_2]$. Thus, $r \in \Omega_q$.
Next, define:
\begin{align}
\begin{aligned}
\alpha:&=-\beta * \left( (\beta^{-1}*\omega)\phi -\beta^{-1}*\omega \right),\\
&=\omega-\beta *r.
\end{aligned}
\end{align}
It remains to show that $l(\alpha)>\tau$. Since $(\beta^{-1}*\omega)\phi=\beta^{-1}*\omega$ on $(-\infty,b_1)$, we deduce that the support of $-\left(\beta^{-1}*\omega \right)\phi+\beta^{-1}*\omega$ is contained in $[b_1,+\infty)$. Then by definition of $\alpha$, the support of $\alpha$ belongs to $[b_1+l(\beta),0]$. Hence $l(\alpha)\ge b_1+l(\beta)$, but by definition of $b_1$, we have $\tau<b_1+l(\beta)$, and then $l(\alpha)>\tau$.
\end{proof}

\begin{proof}[Proof Theorem~\ref{bound_minimal_time_controllability}]
It is obvious that the quasi-reachability (resp. reachability) in finite time $T>0$ implies the quasi-reachability (resp. reachability). 

Let us show the converse. We will prove that the $X^{Q,q}$ reachability implies the $X^{Q,q}$ reachability in time $T>-l(\det(Q))$ for all $T$. Let $T>-l(\det(Q))$ and we assume that a $\Sigma^Q$ pseudo-rational system of order zero is $X^{Q,q}$ reachable. By assumption, for every $\psi \in X^{Q,q}$, we can find an input $\omega \in (\Omega_q)^m$ such that
\begin{equation}
\label{eq_th_bound}
    \pi(Q^{-1}*P*\omega)=\psi.
\end{equation}

Applying Lemma~\ref{lemma_Kamen} with $\beta= \det(Q)$ on each components on $\omega$, we get the existence of $\alpha,r \in (\Omega_q)^m$ such that $\omega=\alpha+\det(Q)*r$ and $l(\alpha)>-T$. We deduce that
\begin{equation}
\label{eq_th_bound2}
    \pi(Q^{-1}*P*\omega)=\pi(Q^{-1}*P*\alpha)+\pi(\mathrm{Adj(Q)}*P*r),
\end{equation}
where $\rm{Adj}(Q)$ is the adjugate matrix of $Q$ and it is equal to $\det(Q)*Q^{-1}$ by a well known result on matrices. Thus, since the support of the function $\mathrm{Adj(Q)}*P*r$ is included in $\mathbb{R}_-$, we get that $\pi(\mathrm{Adj(Q)}*P*r)$ is zero (as an element in $\left(L^q_{\rm loc}(\mathbb{R}_+,\mathbb{R}\right)^d$). We deduce from Equations~\eqref{eq_th_bound}-\eqref{eq_th_bound2} that
\begin{align}
\begin{aligned}
    \pi(Q^{-1}*P*\alpha)&=\pi(Q^{-1}*P*\omega),\\
  &=\psi.
    \end{aligned}
\end{align}
In other words, for every target $\psi$, we have an effective procedure to construct a control $\alpha$ with support in $[-T,0]$ allowing to reach $\psi$. Thus $\Sigma^Q$ is $X^{Q,q}$ reachable in time $T$ for all $T>-l(\det(Q))$.

We only sketch the proof of the $X^{Q,q}$ quasi-reachability implying the $X^{Q,q}$ quasi-reachability in time $T>-l(\det(Q))$ for all $T$. Let $T>-l(\det(Q))$, by assumption, for every $y \in  X^{Q,q}$ there exist a sequence $(u_n)_{n \in \mathbb{N}} \in (\Omega_q)^m$ such that 
    \begin{equation}
    \label{equation_supp_proof1}
        \pi(Q^{-1}*P*u_n) \longrightarrow y,\quad\mbox{in $(L^q_{\rm loc}(\mathbb{R}_+,\mathbb{R}))^m$ as $n \rightarrow +\infty.$}
    \end{equation}

Applying the same strategy of proof as for the reachability, we can find a sequence $(\alpha_n)_{n \in \mathbb{N}} \in (\Omega_q)^m$ with support in $[-T,0]$ such that 
    \begin{equation}
    \label{equation_supp_proof2}
        \pi(Q^{-1}*P*u_n) = \pi(Q^{-1}*P* \alpha_n).
   \end{equation}

   From Equations~\eqref{equation_supp_proof1}-\eqref{equation_supp_proof2}, we get 
    \begin{equation}
    \label{equation_supp_proof3}
        \pi(Q^{-1}*P*\alpha_n) \longrightarrow y,\quad\mbox{in $(L^q_{\rm loc}(\mathbb{R}_+,\mathbb{R}))^m$ as $n \rightarrow +\infty.$}
    \end{equation}
Thus the system $\Sigma^Q$ is $X^{Q,q}$ quasi-reachable in time $T$ for all $T>-l(\det(Q))$.

\end{proof}

\begin{remark}
\label{remark_build1}
Following the proofs of Lemma~\ref{lemma_Kamen} and Theorem~\ref{bound_minimal_time_controllability}, we are able to build a control in finite time $T$ for all $T>-l(\det(Q))$ from a given input $u$ allowing us to reach the same output as $u$.
\end{remark}

\section{Left-coprimeness reachability conditions}
\label{sec4}
The $X^{Q,q}$ quasi-reachability can be fully characterized by an approximate left-coprimeness condition, as proven in \cite[Theorem 4.4]{YamamotoRealization}. However, a similar characterization for the $X^{Q,q}$ reachability (in finite time) of pseudo-rational systems is not available in the literature. In this section, we investigate the relationship between the $X^{Q,q}$ reachability and left-coprimeness condition. We begin by defining the notion of left-coprimeness.

\begin{definition}
\label{def_coprim}
   A $\Sigma^Q$ pseudo-rational system of order zero is \textit{left-coprime} if there exist matrices $R$ and $S$ of appropriate size with entries in $M(\mathbb{R}_-)$ such that   \begin{equation}
    \label{eq:app_cont0bisbis}
        Q*R+P*S
        =
        I_d \delta_0.
    \end{equation}
\end{definition}

We are now ready to show that the left-coprimeness condition implies the $X^{Q,q}$ reachability in finite time. 

\begin{proposition}
\label{prop_suffi}
    If  a $\Sigma^Q$ pseudo-rational system of order zero is left-coprime, then it is $X^{Q,q}$ reachable in time $T>-l(\det(Q))$.
\end{proposition}

\begin{proof}
 In view of Theorem~\ref{bound_minimal_time_controllability}, it remains to prove that a left-coprime $\Sigma^Q$ pseudo-rational system of order zero is $X^{Q,q}$ reachable.
Consider a target output $ \psi \in X^{Q,q}$ and let us define the input given by
\begin{equation}\label{eq:motion-planning}
\omega := S * Q * \psi.
\end{equation}
Since the supports of the elements of \( S*Q \) and \( \psi \) are contained in \( \mathbb{R}_- \) and \( \mathbb{R}_+ \), respectively, it follows that the elements of $\omega$ belong to $M_+(\mathbb{R})$. Moreover, since \( \pi(\omega) = 0 \) due to \( \psi \in X^{Q,q} \), we conclude that the elements of \( \omega \) lie in \( \Omega_q \), and thus \( \omega \in (\Omega_q)^m \). From \eqref{eq:app_cont0bisbis}--\eqref{eq:motion-planning}, we have

\begin{align}
\label{eq_prop3_1}
\begin{aligned}
\pi\left(Q^{-1}*P* \omega\right) &=\pi\left(Q^{-1}*P* S*Q*\psi\right) ,\\
&=\pi\left(Q^{-1}*\left(\delta_0 I_d-Q*R \right)*Q*\psi\right),\\
&=\pi(\psi)-\pi(R*Q*\psi).
\end{aligned}
\end{align} 
Since \( \psi \in X^{Q,q} \), it follows that \( \pi(R*Q*\psi) = 0 \). We deduce from Equation~\eqref{eq_prop3_1} that

\begin{equation}
    \pi\left(Q^{-1}*P* \omega\right)=\pi (\psi)=\psi.
\end{equation}
Thus we reach the target ouput $\psi$ with the input $\omega$, achieving the proof of the proposition.

\end{proof}

\begin{remark}
\label{remark_build2}
Proposition~\ref{prop_suffi} is central as it provides a method for constructing inputs to reach desired targets. In other words, if the system $\Sigma^Q$ is left-coprime and the Radon-valued matrix $S$ is known, we can construct an input to reach the desired target by following the proof of Proposition~\ref{prop_suffi}. Moreover, taking into account Remark~\ref{remark_build1}, it is possible to construct inputs that allows reaching any target in $X^{Q,q}$ within a time $T > -l(\det(Q))$.
\end{remark}

Proposition~\ref{prop_suffi} shows that the left-coprimeness condition is sufficient for the $X^{Q,q}$ reachability of pseudo-rational system of order zero. We will now prove that it is also a necessary condition when $q$ is equal to one.

\begin{theorem}
\label{th_coprm_L1}
A $\Sigma^Q$ pseudo-rational system of order zero is $X^{Q,1}$ reachable in time $T>-l(\det(Q))$ if and only if it is left-coprime.
\end{theorem}

\begin{proof} From Proposition~\ref{prop_suffi}, if a $\Sigma^Q$ pseudo-rational system of order zero is left-coprime, it is $X^{Q,1}$ reachable in time $T>-l(\det(Q))$. Let us show the converse.

Let $T>-l(\det(Q))$. Applying \cite[Lemma 4.3]{yamamoto1989reachability}, we have the existence of a sequence 
$\psi_n=(\psi_{n,1},\dots,\psi_{n,d}) \in (X^{Q,2})^d \subset (X^{Q,1})^d$,  $n \in \mathbb{N}$, such that $ \psi_n \rightarrow \pi (Q^{-1})$ in the distributional sense as $n\to\infty$. In particular, since each column of $Q^{-1}$ is a distribution of order zero, we get the existence of a $C>0$  such that
\begin{equation}
\label{estimate_eq_thm}
    \|\psi_{n,j}(\cdot)\|_{X^{Q,1}_T} \le C,\qquad \forall j \in 1,...,d,\ \forall n \in \mathbb{N}.
\end{equation}

Let us consider $\overline{\Omega}_1$ the subspace of $\Omega_1$ composed of the inputs with support in $[-T,0]$ endowed with the norm $\|\cdot\|_{L^1([-T,0],\mathbb{R})}$. We consider from now on the product norm, denoted by $\|\cdot \|_{(\overline{\Omega}_1)^m}$, on the space $ (\overline{\Omega}_1)^m$. The map from $ (\overline{\Omega}_1)^m$ into $X^{Q,1}$, two Banach spaces with respect to the norms $\|\cdot\|_{(\overline{\Omega}_1)^m}$ and $\|\cdot\|_{X^{Q,1}_T}$, defined by  $u \mapsto \pi(Q^{-1}*P*u) $ is a continuous surjective map. Thus, using the open mapping theorem (\cite[Theorem 12.1, p. 90]{conway2019course}) together with the estimate given in Equation~\eqref{estimate_eq_thm} provide the existence of $M'>0$ and $S_{n,j} \in (\overline{\Omega}_1)^m$ such that
\begin{equation}
\label{estimate_eq_thm2}
    \pi(Q^{-1}*P*S_{n,j})=\psi_{n,j} \quad \mathrm{and} \quad \|S_{n,j}\|_{(\overline{\Omega}_1)^m} \le M',\quad \forall j \in 1,...,d, \quad \forall n \in \mathbb{N}.
\end{equation}

Denoting $S_n=(S_{n,1},\cdots,S_{n,d})$, the weak-compactness of Radon measures, see for instance \cite[Theorem 2 p.55]{evans1992studies}, and the estimate given in the inequality of \eqref{estimate_eq_thm2} imply the existence of a subsequence of $(S_n)_{n\in \mathbb{N}}$ converging toward a radon measure $S$ in a distributional sense, where the elements of $S$ have compact supports in $\mathbb{R}_-$. Then, taking the limit (with respect to the subsequence) in the first equation of \eqref{estimate_eq_thm2} leads to $\pi(Q^{-1}*P*S)=\pi(Q^{-1})$. Let us define $R:=Q^{-1}-Q^{-1}*P*S$. Since $\pi(R)=0$, it follows that the elements of $R$ belongs to $M(\mathbb{R}_-)$. Multiplying $R$ with respect to the convolution on the left by $Q$ allows us to conclude that $\Sigma^Q$ is left-coprime.

\end{proof}

\begin{remark}
  It is important to emphasize that the proof provided for the case \( q = 1 \) does not directly generalize to \( q > 1 \). Specifically, the convergence \( \psi_n \to \pi(Q^{-1}) \) in the distributional sense guarantees the boundedness of \( (\psi_{n,j})_{n \in \mathbb{N}} \) in $X^{Q,1}$ for $j \in \{1,\cdots,d\}$. However, this boundedness does not necessarily extend to $X^{Q,q}$ for \( q > 1 \).
\end{remark}

From Theorem~\ref{th_coprm_L1}, the $X^{Q,1}$ reachability is equivalent to the left-coprimeness of $\Sigma^{Q}$. In complex analysis, the left-coprimeness condition is referred to as a corona problem, with the most celebrated result being Carleson's theorem for holomorphic functions of a single variable that are uniformly bounded in the unit disk, see \cite{carleson1962interpolations}. In the case of matrices with entries in $M(\mathbb{R}_-)$, the problem appears to be complicated, and there is no trivial solution available in the literature. In the current state of knowledge, we are able to provide sufficient criteria for the exact controllability for outputs with sufficient regularity. For two integers $k,n$  and $\mathbb{K}=\mathbb{R}$ or $\mathbb{R}_+$, we denote by $C^k(\mathbb{K},\mathbb{R}^n)$ the space of $k$-times continuously differentiable functions defined on $\mathbb{K}$ with values in $\mathbb{R}^n$. We define a more regular state space for pseudo-rational systems of order zero, i.e.,
for an integer $k$, we let
\begin{equation}
\label{def_state_space}
X_{k}^Q:=\left\{\,y \in  C^k(\mathbb{R}_+,\mathbb{R}^d)\,|\,\pi(Q*\tilde{y})=0,\,\pi(\tilde{y})=y,\, \tilde{y} \in C^k(\mathbb{R},\mathbb{R}^d)  \right\}.
\end{equation}

By a careful inspection, we have that $X_{k}^Q$ is well defined because $\pi(Q*\tilde{y})$ does not depend on the choice of $\tilde{y}$.
To agree with the state space $X^Q_k$, we need to define the reachability for this function space.
\begin{definition}
\label{def:reach_regular}
A pseudo-rational system of order zero $\Sigma^Q$ is $C^k$ reachable if, for every $y \in X_{k}^Q$, there exists a continuous function $u \in C^0(\mathbb{R},\mathbb{R}^m)$ compactly supported in $\mathbb{R}_-$ such that $\pi(Q^{-1}*P*u)=y$.
\end{definition}

We state and prove a sufficient criterion for the $X^Q_k$ reachability in the proposition below.

\begin{proposition}\label{prop-Ckexact}
Consider $\Sigma^Q$ a pseudo-rational system of order zero. If there exists $\alpha>0$ such that, for every $p\in \mathbb{C}$,  
\begin{equation}
    \label{eq_corona}
\inf \left\{ \vertii{g^T\widehat{Q}(p)}+\vertii{g^T \widehat{P}(p)} \suchthat \mbox{$g\in \mathbb{C}^d$, $\|g\|=1$
 } \right\} \ge \alpha,
 \end{equation} then
$\Sigma^Q$ is $C^k$ exactly reachable for some integer $k$.
\end{proposition}

\begin{proof} Assume the existence of two matrices $R$ and $S$, with entries in $\mathcal{E}'(\mathbb{R}_-)$ such that 
\begin{equation}
    \label{eq:app_cont0bisbisbeuh}
        Q*R+P*S
        =
        I_d \delta_0.
    \end{equation}
Let $y \in X^Q_k$ with $k$ the order of the distribution $S$. Consider a continuation $\tilde{y} \in C^k(\mathbb{R},\mathbb{R}^d)$ of $y$ such that $\tilde{y}=y$ on $\mathbb{R}_+$ and the support of $\tilde{y}$ is bounded on the left. Similarly as in the proof of Proposition~\ref{prop_suffi} we can show that, if we define $u:=S * Q * \tilde{y}$, then $u \in C^0(\mathbb{R},\mathbb{R}^m)$ is compactly supported in $\mathbb{R}_-$ and
\begin{equation}
    \pi(Q^{-1}*P*u)=y.
\end{equation}

It remains to prove that, under Equation~\eqref{eq_corona}, we have the existence of two matrices $R$ and $S$ with entries in $\mathcal{E}'(\mathbb{R}_-)$ satisfying \eqref{eq:app_cont0bisbisbeuh}. By the trick of P. Furhmann \cite[proof Theorem 3.1]{fuhrmann1968corona}, it is sufficient to use the one-dimensional corona theorem~\ref{corona_thm}, achieving the proof of the proposition.

\end{proof}

\begin{remark}
Theorem~\ref{th_coprm_L1} and Proposition~\ref{prop-Ckexact} highlighted the difficulty of characterizing reachability with inputs and outputs in function spaces, as well as providing frequency domain criteria for reachability. An alternative approach to characterizing reachability, based on the left-coprimeness condition in more general algebras, involves extending the input and output spaces to include distributions. This approach is presented in \cite{YAMAMOTO201620}. Definition~\ref{def:reach_regular} and Proposition~\ref{prop-Ckexact} go a step further by deriving sufficient reachability criteria from this framework in function spaces, whereas the paper \cite{YAMAMOTO201620} exclusively considers inputs and outputs within distribution spaces.
\end{remark}

Pseudo-rational system of order zero are a framework suited to deal with controllability issues of difference delay systems as shown in \cite{chitour:hal-03827918} and \cite{fueyo-chitour}. As a natural application of the results stated in Sections~\ref{sec_min_time}-\ref{sec4}.

\section{Application to delay systems}
\label{sec:appli_delay7}
Let us consider a controlled difference delay system with distributed delays of the form
\begin{equation}
\label{system_lin_formel2}
 x(t)=\sum_{j=1}^NA_jx(t-\Lambda_j)+\int_0^{\Lambda_N} g(s) x(t-s)ds+Bu(t) , \qquad t \ge 0,
\end{equation}
where, given three positive integers $d$, $m$ and $N$, $g(\cdot)$ belongs to $L^{\infty}([0,\Lambda_N],\mathbb{R}^{d \times d}$) the space of uniformly bounded matrices on $[0,\Lambda_N]$, $A_1,\dotsc,A_N$ are fixed $d\times d$ 
matrices with real entries, the state $x$ and the control $u$ belong 
to $\mathbb{R}^d$ and $\mathbb{R}^m$ respectively, and $B$ is a fixed $d \times m$ matrix 
with real entries. Without loss of generality, the delays $0<\Lambda_1, \dotsc, \Lambda_N$ 
are positive real numbers so that $\Lambda_1< \dotsb <\Lambda_N$. System~\eqref{system_lin_formel2} appears typically in literature to study one dimensional hyperbolic partial differential equations, see \cite{auriol2019explicit}. As shown in \cite{yamamoto1989reachability} for controlled neutral difference equations, we have that controlled delay systems can be construed as pseudo-rational systems. More precisely, we interpret System~\eqref{system_lin_formel2} in the realization formalism described above: 
\begin{equation}
\label{syst_lin_avec_sortie}
\begin{dcases}
 x(t)=\sum_{j=1}^NA_jx(t-\Lambda_j)+\int_0^{\Lambda_N} g(s)x(t-s)ds+Bu(t),&\text{ for 
  $t\ge \inf \supp(u)$}, \\
  x(t)=0,&\text{ for 
  $t<\inf \supp(u)$},\\
y(t)=x(t-\Lambda_N),&\text{ for 
$t \in [0,+\infty)$},
\end{dcases}
\end{equation}
with $u \in (\Omega_{q})^m$. We introduce the two following Radon measures associated to System~\eqref{syst_lin_avec_sortie}:
\begin{equation}
\label{eq:defQ}
\begin{split}
Q & :=\delta_{-\Lambda_N} I_d- \sum_{j=1}^N \delta_{-\Lambda_N+\Lambda_j} A_j- \delta_{-\Lambda_N}*\tilde{g},\\
P & := B \delta_0,
\end{split}
\end{equation}
where $\tilde{g}$ is the extension of $g$ on $\mathbb{R}$ by zero on the set $(-\infty,0) \cup (\Lambda_N,+\infty)$. Thanks to Proposition~\ref{theorem_pseudo_rational}, we have that System~\eqref{syst_lin_avec_sortie} is a $\Sigma^Q$ pseudo-rational system of order zero. As a straighforward consequence of Theorem~\ref{bound_minimal_time_controllability} we have the following equivalence between the reachability notions.

\begin{theorem}
\label{theorem_reachability_dis_delay}
System~\eqref{syst_lin_avec_sortie} is $X^{Q,q}$ quasi-reachable (resp. reachable) if and only if it is  $X^{Q,q}$ quasi-reachable (resp. reachable) in time $\widetilde{T}$ for all $\tilde{T}>d \Lambda_N$. In particular, an upper bound on the minimal time of quasi-reachability (resp. reachability) is given by $d \Lambda_N$.
\end{theorem}

\begin{proof}
  Since the support of $\det(Q)$ belongs to $[-d \Lambda_N,0]$, we have that the quasi-reachability (resp. reachability) is equivalent to the quasi-reachability (resp. reachability) in time $T$ for all $T>d \Lambda_N$ from Theorem~\ref{bound_minimal_time_controllability}. 
\end{proof}

\begin{remark}
The result of Theorem~\ref{theorem_reachability_dis_delay} was established in \cite{chitour:hal-03827918} for the more restrictive case where the pseudo-rational systems considered were difference delay systems with a finite number of delays (\( g \equiv 0 \) in \eqref{system_lin_formel2}), as described in \eqref{eq:diff_delay_without}. The proof in \cite{chitour:hal-03827918} was complex and highly computational, whereas the proof of Theorem~\ref{bound_minimal_time_controllability}, from which Theorem~\ref{theorem_reachability_dis_delay} is derived, is much simpler.

\end{remark}

Secondly, we are able to exploit the left-coprimeness condition to obtain reachability results for System~\eqref{syst_lin_avec_sortie}. We introduce some notations to present the results. The Laplace transform of $Q$ and $P$ are given by 
\begin{align}
\begin{aligned}
\widehat{Q}(p)&=e^{p \Lambda_N}\left(I_d-\sum_{j=1}^N A_j e^{-p \Lambda_j}-\int_0^{\Lambda_N} g(s) e^{-ps}ds \right),\quad \forall p \in \mathbb{C},\\
\widehat{P}(p)&=B,\quad \forall p \in \mathbb{C}.
\end{aligned}
\end{align}

We denote by \( \overline{\widehat{Q}(\cdot)} \) the closure of the holomorphic map \( \widehat{Q}(\cdot) \), and by \( \rank \left[ F, B \right] \) the rank of the concatenated matrices \( F \in \overline{\widehat{Q}(\cdot)} \) and \( B \).

\begin{theorem}
\label{th:main_result_dist}
Consider the assertion:
\begin{enumerate}[(a)]
    \item \label{item1bis}$\displaystyle \rank \left[F,B\right]=n$ for all $p \in \mathbb{C}$ and $F\in \overline{\widehat{Q}(\cdot)}$.
\end{enumerate}
If Item~\ref{item1bis} is satisfied, then System~\eqref{syst_lin_avec_sortie} is $C^k$ reachable for some integer $k$. Conversely, if System~\eqref{syst_lin_avec_sortie} is $X^{Q,1}$ reachable then it satisfies Item~\ref{item1bis}.
\end{theorem}

\begin{proof}
    First, we prove that Item~\ref{item1bis} implies the assumption of Proposition~\ref{prop-Ckexact}. The proof proceeds by contradiction, assuming that Equation~\eqref{eq_corona} is not satisfied. Thus we would have two sequences $(g_n)_{n \in \mathbb{N}} \in \mathbb{C}^d$ and $(p_n)_{n \in \mathbb{N}} \in \mathbb{C}$ such that $\|g_n\|=1$, $g_n^T \widehat{Q}(p_n) \to 0$ and $g_n^T B \to 0$ when ${n\to+\infty }$. We have three cases:
    \begin{enumerate}[i)]

        \item If $\Re(p_n) \to - \infty$ when ${n\to+\infty }$ then by the Lebesgue dominated convergence theorem we would have that $\lim_{n\to+\infty}\widehat{Q}(p_n)= -A_N$. Furthermore, since $(g_n)_{n \in \mathbb{N}}$ is a bounded sequence in $\mathbb{C}^d$, we can, up to a subsequence, find a $g \in \mathbb{C}^d$ such that $\|g \|=1$ and $\lim_{n \to + \infty} g_n= g$. Thus, we have $-g^TA_N=0$ and $g^T B=0$ which is an inconsistency with Item~\ref{item1bis};
        
        \item Since $\widehat{Q}(p)$ is equivalent to $e^{p \Lambda_N}I_d$ when the real part of $p$ tends to $+\infty$, it is not possible to have $\Re(p_n) \to + \infty$ when $n\to+\infty $. 
        
        \item Assume that the real part of $(p_n)_{n\in \mathbb{N}}$ is uniformly bounded. We have that $\widehat{Q}(\cdot)$ is bounded on all bounded vertical strip so that, up to a subsequence, there exists $F$ such that $\lim_{n \to +\infty} \widehat{Q}(p_n) = F$. We can also, up to a subsequence, find a $g \in \mathbb{C}$ such that $g_n {\longrightarrow} g$ and $\|g\|=1$ when $n \to + \infty$. We deduce that $g^T F=0$ and $g^T B=0$ which is a contradiction of Item~\ref{item1bis}.
    \end{enumerate}
   Thus, the assumption of Proposition~\ref{prop-Ckexact} is satisfied, allowing us to apply the result of the same proposition, which leads to the conclusion of the first part of the theorem.

    Assume now that System~\eqref{syst_lin_avec_sortie} is $X^{Q,1}$ reachable. Then it is a left-coprime system by Theorem~\ref{th_coprm_L1}. We apply the Laplace transform in Equation~\eqref{eq:app_cont0bisbis} and multiplying by $g^T\in \mathbb{C}^d$ on the left, we get
    \begin{equation}
    \label{eq:app_cont0bisbisbis}
        g^T \widehat{Q}(p)\widehat{R}(p)+g^T B\widehat{S}(p)
        = g^T, \quad \forall p \in \mathbb{C},
    \end{equation} 
    for some matrices $R$ and $S$ with Radon measure entries in $M(\mathbb{R}_-)$. If Item~\ref{item1bis} was not satisfied, then there would exists a nonzero $g\in \mathbb{C}^d$, a matrix $F$ and a sequence $(p_n)_{n \in \mathbb{N}} \in \mathbb{C}$, with a real part uniformly bounded above, such that $g^T F=0$, $g^T B=0$ and $\lim_{n \to + \infty}\widehat{Q}(p_n) = F$. From \eqref{eq:app_cont0bisbisbis}, we get 
    \begin{equation}
\label{eq:app_cont0bisbisbisbis}
        g^T \widehat{Q}(p_n)\widehat{R}(p_n)+g^T B\widehat{S}(p_n)
        = g^T, \quad \forall n \in \mathbb{N}.
    \end{equation} 
     By the estimate of the Paley-Wiener-Schwartz theorem, see Theorem~\ref{th_Paley}, we get that the maps $\widehat{R}(\cdot)$ and $\widehat{S}(\cdot)$ are uniformly bounded in the half space $\{p \in \mathbb{C}|\, \Re(p) \le \alpha\}$ for each $\alpha \in \mathbb{R}$. Thus we deduce that $\lim_{n \to + \infty} g^T B\widehat{S}(p_n) = 0 $ and $ \lim_{n \to + \infty} g^T \widehat{Q}(p_n)\widehat{R}(p_n)= 0 $. Taking the limit in \eqref{eq:app_cont0bisbisbisbis} when $n$ tends to $+\infty$, we would obtain that $g^T=0$.
    It is a contradiction so that Item~\ref{item1bis} holds true.
\end{proof}

\begin{remark}
    Item~(a) of Theorem~\ref{th:main_result_dist} is a Hautus test that is very similar to the one given in \cite{hautus1969controllability} for ordinary differential equations.
\end{remark}

\begin{remark}
Theorem~\ref{theorem_reachability_dis_delay} and Theorem~\ref{th:main_result_dist} were known for difference delay systems without distributed delays; see \cite{chitour:hal-03827918}. The pseudo-rational of order zero systems framework developed in this article extended the results of \cite{chitour:hal-03827918} in the more complicated case where distributed delays occur in difference delay systems. 
\end{remark}

\section*{Acknowledgement}
The author sincerely thanks the anonymous reviewers for their careful reading and insightful suggestions, which significantly contributed to the improvement of the manuscript.

\bibliographystyle{elsarticle-num} 
\bibliography{ifacconf}

\appendix

\section{Invertibility of a Radon measure linked to a difference delay systems with distributed delays}

We show the invertibility of the matrix $Q$ associated with difference delay systems with distributed delays.

\begin{proposition}
\label{theorem_pseudo_rational}
   The radon measure $Q$ defined in \eqref{eq:defQ} is invertible over $M_+(\mathbb{R})$.
\end{proposition}
\begin{proof}
 First, by a simple property of the space $L^1(\mathbb{R},\mathbb{R}^{d \times d})$ and the definition of $\tilde{g}$, we have for all $t_0 \in [0,\Lambda_N]$:
 \begin{equation}
     \int_{-\infty}^{\infty} \|\tilde{g}(t)\|dt=\int_{0}^{\Lambda_N} \|\tilde{g}(t)\|dt=\int_{0}^{t_0} \|\tilde{g}(t)\|dt+\int_{t_0}^{\Lambda_N} \|\tilde{g}(t)\|dt.
 \end{equation}
 Taking $t_0$ enough small, we can have the quantity $\int_{0}^{t_0} \|\tilde{g}(t)\|dt$ strictly smaller than one. In another words, for some $\epsilon \in [0,\Lambda_1)$, we can decompose $\tilde{g}$ as
    \begin{equation}
    \label{eq:th_pseudo_rat1}
   \tilde{g} =\tilde{g}_1+\tilde{g}_2,
     \end{equation}
    where $\tilde{g}_1,\tilde{g}_2 \in $, $\vertii{\tilde{g}_1}_1<1$, $\mathrm{supp}\, \tilde{g}_1 \subseteq [0,\Lambda_1-\epsilon]$ and $\mathrm{supp}\, \tilde{g}_2 \subseteq [\Lambda_1-\epsilon,\Lambda_N]$. We can also write $Q$ as
\begin{equation}
\label{eq:proof_pseudo}
    Q=I_d\delta_{-\Lambda_N}*(I_d\delta_0-\tilde{g}_1)*(I_d \delta_0+G),
\end{equation}
where
\begin{equation}
\label{eq:proof_pseudo1}
    G=-(I_d \delta_0-\tilde{g}_1)^{-1}*\left( \sum_{j=1}^N \delta_{\Lambda_j} A_j+\tilde{g}_2 \right) \mbox{ and } (I_d \delta_0 -\tilde{g}_1)^{-1}=I_d \delta_0+\sum_{j=1}^{+\infty} \tilde{g}_1^{\ast j}.
\end{equation}
The notation $\phantom{}^{*j}$ introduced in \eqref{eq:proof_pseudo1} denotes the convolution product repeated $j \in \mathbb{N}^*$ times.
Thus we have that the Radon measure $Q^{-1}=I_d \delta_{\Lambda_N}*(I_d \delta_0+G)^{-1}*(I_d \delta_0-\tilde{g}_1)^{-1}$ is an inverse of $Q$ over $M_+(\mathbb{R})$ with
\begin{equation}
\label{eq:proof_pseudo2}
   (I_d\delta_0+G)^{-1}=I_d \delta_0+\sum_{j=1}^{+\infty}(-1)^j G^{\ast j}.
    \end{equation}

\end{proof}

\section{Paley-Wiener-Schwartz theorem}

We state the Paley-Wiener-Schwartz for distributions with compact support in $\mathbb{R}_-$, see for instance \cite[Theorem 7.3.1]{lars1990analysis}.

\begin{theorem}(Paley-Wiener-Schwartz)
\label{th_Paley}
A holomorphic function $f(\cdot)$ is the Laplace transform of a distribution $\phi \in \mathcal{E}'(\mathbb{R}_-)$ if and only if, for some $C,a \in \mathbb{R}_+$ and $m$ integer, it satisfies the following inequalities
\begin{align}
\begin{aligned}
    \abs{f(s)} &\le C (1+\abs{s})^m e^{a \Re(s)},\quad s \ge 0,\\
    \abs{f(s)} &\le C (1+\abs{s})^m, \quad s \le 0.
    \end{aligned}
\end{align} 
In particular, if $\hat{f}(\cdot)$ is the Laplace transform a Radon measure $f \in M(\mathbb{R}_-)$ then we have the following estimate for some positive real $\tilde{a}$ and $\tilde{C}$
\begin{align}
\begin{aligned}
    \abs{f(s)} &\le \tilde{C}  e^{\tilde{a} \Re(s)},\quad s \ge 0,\\
    \abs{f(s)} &\le \tilde{C} , \quad s \le 0.
    \end{aligned}
\end{align}

\end{theorem}

\section{A corona theorem for distributions with compact support in $\mathbb{R}_-$}

We recall a corona theorem, initially established in \cite{Yamamoto_Willems} for the scalar case with two elements, and later extended to multiple elements by \cite{maad2011generators}, who formulated the result for \( \mathcal{E}'(\mathbb{R}_+) \), whereas Yamamoto and Willems originally formulated it for \( \mathcal{E}'(\mathbb{R}_-) \). We do not use in the full generality this theorem but an easier subversion which is just a straighforward application of the paper of Hörmander \cite{hormander_paper}. 

\begin{theorem}
\label{corona_thm}
    Let $m$ an integer and $f_1,...,f_m \in \mathcal{E}'(\mathbb{R}_-)$ such that there exists $\alpha>0$ satisfying 
\begin{equation}
    \sum_{k=1}^m \abs{\hat{f_k}(p)}\ge \alpha,\quad \forall p \in \mathbb{C}.
\end{equation}
Then there exist $g_1,...,g_m \in \mathcal{E}'(\mathbb{R}_-) $ such that 
\begin{equation}
    \sum_{k=1}^m f_k* g_k=\delta_0.
\end{equation}
\end{theorem}

\end{document}